\documentclass{emsprocart}

\contact[sasajidegozaimath@gmail.com]
{Takahisa Sasajima, Nakagyou-ku Koromono-tana-dori Oike-agaru Shimo-myoukaku-ji 192-3, Kyoto 604-0024, Japan}

\contact[ohmoto@math.sci.hokudai.ac.jp]
{Toru Ohmoto, Department of Mathematics, Hokkaido University, Sapporo 060-0810, Japan}

\usepackage{tikz-cd}
\usepackage{amsmath,amssymb}
\usepackage{mathrsfs}
\usepackage{makecell}
\usepackage{graphicx}

\newtheorem{thm}{Theorem}[section]
\newtheorem{cor}[thm]{Corollary} 
\newtheorem{lem}[thm]{Lemma}

\newtheorem{rem}[thm]{Remark} 
\newtheorem{exam}[thm]{Example} 
 
\newtheorem{definition}[thm]{Definition}

\newcommand{\A}{\mathcal{A}}
\newcommand{\K}{\mathcal{K}}

\newcommand{\F}{\mathcal{F}}
\newcommand{\Ost}{\mathcal{O}}
\newcommand{\Proj}{\mathbb{P}}
\newcommand{\C}{\mathbb{C}}
\newcommand{\Z}{\mathbb{Z}}

\newcommand{\R}{\mathbb{R}}

\newcommand\Dual{\textup{\textrm{Dual}\,}}

\def\qed{\hfill $\Box$}

\def\t{\noindent}

\def\bc{\mbox{\boldmath $c$}}

\def\b0{\mbox{\boldmath $0$}}

\title[Thom polynomials in $\A$ classification]
{Thom polynomials in $\A$-classification I:   
counting singular projections of a surface}
\author[T.~Sasajima, T.~Ohmoto]{Takahisa Sasajima and Toru Ohmoto 
\thanks{
The authors are grateful to L. Feh\'er and Y. Kabata for useful discussions. 
The second author thanks H. Hauser and University of Vienna 
for their hospitality during the writing of the final form of this paper. 
This work was supported by JSPS KAKENHI Grant Numbers 24340007 and 15K13452. }}

\begin{document}
\emph{Dedicated to Professor Piotr Pragacz on the occasion of his 60th birthday.}

\begin{abstract} 
We study universal polynomials of characteristic classes associated to 
the $\A$-classification of  map-germs  $(\C^2,0) \to (\C^n, 0)$ $(n=2,3)$, 
that enable us to systematically generalize 
enumerative formulae 
in classical algebraic geometry of projective surfaces  
in $3$ and $4$-spaces. 
\end{abstract}

%
%


%
%
\maketitle

%


\section{Introduction}
In this paper, we study universal polynomials of characteristic classes associated to 
$\A$-classification theory of holomorphic map-germs $(\C^m,0) \to (\C^n,0)$ 
in low dimensions. 
Our aim is to apply them to classical enumerative geometry. 
In the singularity theory of maps, 
it is natural to identify two map-germs  if they coincide through 
suitable local coordinate changes of the source and the target, 
that is called the \textit{$\A$-equivalence} of map-germs. 
A map-germ is called to be {\it stable} if any deformation of the germ is 
(parametrically) $\A$-equivalent to the trivial one, and 
a holomorphic map $f: X \to Y$ between complex manifolds 
is {\it locally stable} if the germ of $f$ at any point $x \in X$ is stable. 
To each $\A$-singularity type of map-germs, 
one can assign a unique universal polynomial in Chern classes, 
called the \textit{Thom polynomial} 
\cite{FR02, FR, HK, Kaz03, Ohmoto14, Porteous, Rimanyi, Thom, Kleiman}; 
it is defined by the $\A$-equivariant Poincar\'e dual to the $\A$-orbit closure 
in the  space of all map-germs $(\C^m,0) \to (\C^n,0)$. 
As a typical example, 
for the stable singularity type $S_0: (x, y^2, xy)$ (crosscap) of 
map-germ $(\C^2,0)\to(\C^3,0)$, 
the Thom polynomial is given by 
$$
Tp_\A(S_0)=\bar{c}_2=c_1^2-c_2-c_1c_1'+c_2'
$$
where $\bar{c}=1+\bar{c}_1+\bar{c}_2+\cdots =(1+c'_1+\cdots)(1+c_1+\cdots)^{-1}$. 
This counts the number of crosscaps 
in a given locally stable map from a surface into a $3$-fold -- 
precisely saying, 
for any locally stable map $f:X\to Y$ with $\dim Y=\dim X+1$, 
$Tp_\A(S_0)$ evaluated by $c_i=c_i(TX)$ and $c'_j=f^*c_j(TY)$ 
expresses the homology class of the $S_0$-singularity locus  of $f$ in $H^*(X)$. 

Thom polynomials of stable singularity types of maps have been well studied 
(see Feh\'er-Rim\'anyi \cite{FR02, FR}, Kazarian \cite{Kaz03}), 
while those of {\it unstable} $\A$-singularity types have 
been little known so far. 
Our first purpose is to compute Thom polynomials 
for the $\A$-classification 
of (unstable) $\A$-simple map-germs
\begin{center}
$(\C^2,0)\to (\C^2,0)$, $(\C^2,0)\to (\C^3,0)$ and $(\C^3,0)\to (\C^3,0)$ 
\end{center}
in low codimensions 
 (Theorem \ref{main_thm}). 
 To compute the precise form of Thom polynomials, 
we use an effective method due to R. Rim\'anyi  \cite{Rimanyi}, 
called \textit{restriction} or \textit{interpolation method}, 
and indeed it works well for singularities of low codimension (Remark \ref{quasi_homog}). 
For instance, for $\A$-types  
$$B_k\;  (x, y^2, x^2y+y^{2k+1}): (\C^2, 0) \to (\C^3, 0)$$
($k\ge 1$), 
their Thom polynomials are given by, e.g. 
\begin{align}
Tp_\A(B_1)&=&
 -3 c_1^3 + 4 c_1 c_2 + 4 c_1^2 c_1' - 2 c_2 c_1' - c_1 c_1'^2 - 3 c_1 c_2' + c_1' c_2' + c_3',\nonumber \\ 
 Tp_\A(B_2) &=& 
11 c_1^4 - 22 c_1^2 c_2 + c_2^2 - 17 c_1^3 c'_1 + 21 c_1 c_2 c'_1+ 7 c_1^2 {c'_1}^2 - 5 c_2 {c'_1}^2 \nonumber\\
&&- c_1 {c'_1}^3 + 10 c_1^2 c'_2- 5 c_1 c'_1 c'_2 + {c'_1}^2 c'_2 - {c'_2}^2 - 10 c_1 c'_3 + 4 c'_1 c'_3. 
\end{align}
Such a local singularity of maps generically appears in families of maps, 
and hence the associated characteristic class 
expresses the degeneracy locus of that type; 
this is regarded as a {\it non-linear version} of 
the famous Thom-Porteous formula 
for degeneracy loci of vector bundle morphisms \cite{FP}. 
Precisely, it is described as follows. 
Let $E$ and $F$ be holomorphic fiber bundles 
over a complex manifold $X$ with fiber $\C^m$ and $\C^n$, respectively. 
Let $\varphi$ be a holomorphic map 
from an open neighborhoods of the zero section $Z_E$ to $F$ 
so that $\varphi$ identifies the zero sections, $\varphi(Z_E)=Z_F$, 
and the following diagram commutes: 
$$
\begin{tikzcd}
(E, Z_E) \arrow{rd}\arrow{rr}{\varphi}
& &
(F, Z_F) \arrow{ld}  \\
& X  & 
\end{tikzcd}
$$
That is,  $\varphi$ defines a family of holomorphic map-germs 	
$\varphi_{x}:({E}_{x},0)\rightarrow ({F}_{x},0)$
parametrized by points of $X$ analytically. 
Let $\eta(\varphi)$ denote 
the locus consisting of $x \in X$ 
at which the germ $\varphi_x$ is $\A$-equivalent to 
a given type $\eta: (\C^m, 0) \to (\C^n,0)$, 
and it becomes a locally closed subvariety of $X$. 
Then, for an appropriately generic $\varphi$, 
the closure of the $\eta$-locus is universally expressed by 
$$
\Dual[\overline{\eta(\varphi)}]=Tp_\A(\eta)(c(E), c(F))\in H^*(X),
$$
where $c_i=c_i(E)$ and $c_j'=c_j(F)$ are 
the Chern classes of $E$ and $F$ respectively 
(Theorem \ref{thm:A-classification}).

We apply these universal polynomials to enumerative problems 
in classical algebraic geometry. 
As it is dealt in old literatures, 
e.g. Salmon \cite{Salmon}, Baker \cite{Baker}, Semple-Roth \cite{SR}, 
or reformulated within modern intersection theory 
in Piene \cite{Piene} (also Kleiman \cite{Kleiman}), 
it is a typical question to seek for formulae 
of numerical characters of a given projective surface; 
for instance, polar classes are indeed defined by 
the critical loci of general linear projections, 
so those should directly be connected with Thom polynomials \cite{Piene, SO}. 
Furthermore, 
richer structures in local extrinsic geometry of embedded surfaces  
in projective or affine spaces  
have been re-explored as an application of $\A$-classification theory of map-germs 
\cite{Arnold, DK, Kabata, Mond82, Mond, Anna, 
NT, Platonova, Rieger, SKSO, UV, YKO}; 
for instance, the contact of a line with a given surface  
is measured by the $\A$-equivalence type of its projection along the line. 
From this context, 
we study a counting problem of lines 
having some prescribed contact with an embedded surface. 
First, we rediscover classical formulae of Salmon-Cayley-Zeuthen    
for smooth surfaces in $\Proj^3$ 
and generalize them to the case of surfaces with ordinary singularities 
(Theorem \ref{dCT}, Remark \ref{Kulikov}). 
The same approach is also applied to smooth surfaces in $\Proj^4$ 
 (Theorem \ref{dCT2}); 
seemingly, this case was not well examined in old literatures. 
Indeed, the contact of lines has been studied rather recently 
by using the $\A$-classification of map-germs $(\C^2, 0) \to (\C^3, 0)$ 
in D. Mond's thesis \cite{Mond82}, see also \cite[Appendix]{DK}. 
Let $X$ be an appropriately generic surface in $\Proj^4$; 
the projection along almost all tangent lines $l$ is of crosscap $S_0$, 
while along the asymptotic line, 
the projection is of type $B_1(=S_1)$ or worse;  
in particular, the $B_2$-singularity and the $H_2$-singularity  
occur on curves on $X$ 
and these two curves meet each other tangentially at some discrete points 
where the $P_3$-singularity appears in the projection 
(that is an analogy to the {\it cusp of Gauss} 
on a surface in $3$-space). 
Our concern is to compute the degree of such special loci on $X$. 
To each pair $(x,l)$ with $x \in X$ and $x \in l$, 
we associate the germ at $x$ of 
the projection of $X$ along $l$ 
to a hyperplane $ \Proj^3_{(x,l)}$ transverse to $l$ 
passing through $x$, 
that yields a family of holomorphic map-germs 
$$\varphi: (X, x) \to (\Proj^3_{(x,l)}, x)$$ 
over the flag manifold of pairs $(x, l)$. 
We then apply $Tp_\A(\eta)$ to the family $\varphi$, 
and push it down to the ambient $\Proj^4$ to get the desired enumeration. 
In particular, if $X$ is 
a smooth complete intersection of degrees $d_1$ and $d_2$ in $\Proj^4$, 
the degree given by $Tp_\A(B_1)$ is equal to $2d_1d_2$ 
(that is obvious, because at every general point $x \in X$ 
there are exactly two asymptotic lines); 
the degrees of loci of types $B_2, H_2, P_3$ are now easily given by 
\begin{align*}
&B_2: d_1d_2(16d_1+16d_2-55),\quad 
H_2: d_1d_2(6d_1+6d_2-20), \\
&P_3: d_1d_2(18{d_1}^2+18{d_2}^2+51d_1d_2-160d_1-160d_2+280).
\end{align*}

Our method produces a bunch of enumerative formulae of this kind,  
and actually most of those are new. 
Rather than such formulae, 
the main point is to propose a general framework based on 
singularity theory for certain enumerative problems. 
We emphasize that there are a number of different objects  
in projective algebraic geometry 
to which the $\A$-classification of map-germs can suitably be applied, 
e.g.,  osculating planes of a space curve, 
subspaces having a prescribed contact with a hypersurface, 
singular points of a scroll, 
singularities of the focal surface associated to a line congruence or a line complex, 
and so on; 
it would be promising to develop classical works on enumerations 
of such objects  by means of our universal polynomials. 


\section{Thom polynomials in $\A$-classification}

\subsection{Universal classes of  singularity loci of prescribed local types}
Let $\pi_E: E \to X$ and $\pi_F: F \to X$ 
be holomorphic fiber bundles over a complex manifold $X$ 
with fibers $\C^m$, $\C^n$ and zero sections $Z_E$, $\Z_F$, respectively.  
A \textit{holomorphic family of map-germs} 
$\varphi: (E, Z_E)\rightarrow (F, Z_F)$ 
is defined to be the germ at $Z_E$ 
of a holomorphic map $\varphi: U \to F$ 
from an open neighborhood $U$ of $Z_E$  to $F$ 
so that $\pi_F\circ \varphi=\pi_E$ and  $\varphi: Z_E \simeq Z_F$. 
In fact, 
a weaker notion is sufficient for our purpose; 
we say that 
$\varphi$ a {\it continuous family of holomorphic map-germs} if 
$\varphi: U \to F$ is continuous and the restriction $\varphi_x: E_x \cap U \to F_x$ 
to each fiber is holomorphic. 
For simplicity, we work with holomorphic families of map-germs,  
unless specifically mentioned. 

\begin{exam}\upshape
A holomorphic map $f:X\rightarrow Y$ ($m=\dim X$, $n=\dim Y$) 
canonically induces a holomorphic family $\varphi$ of map-germs from $E=TX$ to $F=f^*TY$ 
which corresponds to the map  
$$
X\ni x\longmapsto f_x: (X,x)\rightarrow (Y,f(x)). 
$$
It may be given by the germ at the diagonal of $f\times f: X \times X \to Y\times Y$ 
with projections to the second factors.  
\end{exam}

\begin{definition}\upshape
Let $\eta$ be an $\A$-type of map-germs $(\C^m,0)\rightarrow (\C^n,0)$. 
For a family $\varphi: (E, Z_E)\rightarrow (F, Z_F)$ over $X$, 
we define the \textit{$\eta$-singularity locus of $\varphi$} by 
$$
	\eta(\varphi):=\{\,x\in X\ |\ \varphi_{x}\sim_{\A}\eta\,\}.
$$
\end{definition}

Let $J^r(E,F)$ denote the jet space of holomorphic map-germs 
$({E}_x, 0) \to ({F}_x, 0)$ over all $x\in X$, 
that is, 
$$J^r(E,F):=\bigoplus_{1\leq j\leq r}\textrm{Sym}^j(E^*)\otimes F.$$ 
Given an $\A^r$-invariant subset $\eta \subset J^r(m,n)$, there is 
an associated subbundle $\eta(E,F)$ in $J^r(E,F)$. 
The \textit{$r$-jet extension $j^r\varphi: X \to J^r(E,F)$} is naturally defined by 
assiging $j^r(\varphi_x)(0)$ to $x \in X$, and then 
put $\eta(\varphi):=j^r\varphi^{-1}(\eta(E,F))$. 
We say that $\varphi$ is \textit{admissible} with respect to $\eta$  if 
the equality  
$$\Dual[\overline{\eta(\varphi)}]=j^r\varphi^*(\Dual[\overline{\eta(E,F)}])$$
holds. 
For instance, when $j^r\varphi$ is transverse to  
(a Whitney stratification of) the closure $\overline{\eta(E,F)}$, 
$\varphi$ is admissible.

\begin{thm}[Thom polynomials in $\A$-classification 
\cite{Thom, HK}]\label{thm:A-classification}
For an $\A$-type $\eta$ of $(\C^m,0)\rightarrow (\C^n,0)$ with finite codimension $s$, 
there is a unique class 
$$Tp_{\A}(\eta)(\bc,\bc')\in\Z[c_1,\cdots,c_m,c'_1,\cdots,c'_n]$$
such that 
$$
\Dual[\overline{\eta(\varphi)}]=Tp_{\A}(\eta)(c(E),c(F))\in H^{2s}(X)
$$
for any admissible continuous family 
$\varphi: (E, Z_E)\rightarrow (F, Z_F)$  
of holomorphic map-germs 
with fibers $\C^m$ and $\C^n$. 
\end{thm}

\begin{proof} 
For the sake of completeness, we briefly give a proof, 
although it is straightforward from the definition \cite{HK}. 
There are the classifying maps $\rho_1: X \to BGL(m)$ and $\rho_2: X \to BGL(n)$ 
for linealized bundles (as topological vector bundles) 
$E \to X$ and $F \to X$, respectively, so that the pullback of 
Chern classes $\bc:=c(\xi_1)$ and $\bc'=c(\xi_2)$ of 
the universal vector bundles $\xi_i\; (i=1,2)$ are just $c(E)$ and $c(F)$.  
Put $B=BGL(m) \times BGL(n)$,   
and let $\pi_B: J^r(\xi_1, \xi_2)\to B$ denote 
the universal fiber bundles with fiber $J^r(m, n)$ 
and structure group $GL(m) \times GL(n) \sim \A^r_{m, n}$. 
The jet bundle  has the subbundle $\eta(\xi_1, \xi_2)$ with fiber $\eta$; then 
we can define 
$$Tp_\A(\eta)(\bc, \bc'):=\Dual\, [\overline{\eta(\xi_1, \xi_2)}]$$ 
in the cohomology which is canonically isomorphic to 
$$
H^*(B)\cong  \Z[c_1,\cdots,c_{m}]\otimes\Z[c'_1,\cdots,c'_{n}] =\Z[\bc, \bc']
$$
via the pullback of the projection $\pi_B$. 
Put $\rho=\rho_1 \times \rho_2: X \to B$, and $\bar{\rho}$ the bundle homomorphism over $\rho$. 
By the above construction, 
$$
\Dual\, [\overline{\eta(E,F)}]=\bar{\rho}^*\Dual\, [\overline{\eta(\xi_1, \xi_2)}]. 
$$
Hence 
\begin{align*}
\Dual\,[\overline{\eta(\varphi)}]
&= j\varphi^*(\Dual\, [\overline{\eta(E,F)}])
=j\varphi^*\bar{\rho}^*\Dual\, [\overline{\eta(\xi_1, \xi_2)}] \\
&=\rho^*Tp_\A(\eta)(\bc, \bc')
=Tp_\A(\eta)(c(E), c(F)).
\end{align*}
This completes the proof. 
\end{proof}

\

It is well-known that 
$f: X \to Y$ is \textit{locally stable}  if and only if 
the jet extension $jf$ is transverse to any $\A$-orbits; 
it is also equivalent to that $jf$ is transverse to any $\K$-orbits. 
Thom polynomials for $\K$-orbits (or $\A$-orbits of stable germs) 
have particularly simpler forms as follows. 
Let $\bar{c}_i$ denote the $i$-th term of $(1+c'_1+c'_2+\cdots)(1+c_1+c_2+\cdots)^{-1}$.

\begin{thm}[Thom polynomials for stable singularities, cf. \cite{FR02, Kaz03, Ohmoto14}]
For a stable singularity type $\eta: (\C^m, 0) \to (\C^m,0)$, 
$Tp_{\A}(\eta)$ is a polynomial in  the quotient Chern class $\bar{c}_i$; 
we write it simply by $Tp(\eta)$. 
In particular, for any locally stable map $f: X \to Y$, 
the $\eta$-singularity locus is expressed by the polynomial in 
$\bar{c}_i=c_i(f)=c_i(f^*TY-TX)$. 
\end{thm}

\begin{exam}\upshape
For stable singularity types of $(\C^n, 0) \to (\C^n,0)$ (equidimensional maps) 
of codimension $\le 4$, $Tp$ are 
\begin{align*}
Tp(A_1)&=\bar{c}_1,\\
Tp(A_2)&=\bar{c}_1^2+\bar{c}_2,\\
Tp(A_3)&=  \bar{c}_1^3+3\bar{c}_1\bar{c}_2+\bar{c}_3,\\
Tp(A_4)&= \bar{c}_1^4+6\bar{c}_1^2\bar{c}_2+2\bar{c}_2^2+9\bar{c}_1\bar{c}_3+6\bar{c}_4,\\
Tp(I_{2,2})&=\bar{c}_2^2-\bar{c}_1\bar{c}_3.
\end{align*}
Also for stable singularities of $(\C^n, 0) \to (\C^{n+1}, 0)$, here are some examples:
\begin{align*}
Tp(A_1)&=\bar{c}_2,\\
Tp(A_2)&=\bar{c}_2^2+\bar{c}_1\bar{c}_3+\bar{c}_4.
\end{align*}
\end{exam}

\subsection{$\A$-classification of map-germs and main results}

In late 80's and 90's, there have been done the $\A$-classifications of 
map-germs in some particular dimensions. 
First we quote part of the classification results below. 

\begin{itemize}
\item 
Table \ref{table_2to2}:  $(\C^2,0)\rightarrow(\C^2,0)$ (Rieger \cite{Rieger})
\item 
Table \ref{table_2to3}:  $(\C^2,0)\rightarrow(\C^3,0)$  (Mond \cite{Mond, Mond82})
\item 
Table \ref{table_3to3}: $(\C^3,0)\rightarrow(\C^3,0)$  (Marar-Tari \cite{MararTari}, Hawes \cite{Hawes} etc)
\end{itemize}

\begin{table}
		\centering
        	\begin{tabular}{|c|c|l|}\hline 
		        Type &  $\A$-codim & Normal form \\ \hline \hline
		        Regular  & 0 & $(x,y)$ \\ 
		        Fold  & 1 & $(x,y^2)$ \\ 
		        Cusp  & 2 & $(x,xy+y^3)$ \\ \hline
		        Lips/Beaks  & 3 & $(x,y^3+x^2y)$ \\ 
		        Swallowtail  & 3 & $(x,xy+y^4)$ \\ 
		        Goose & 4 & $(x,y^3+x^3y)$\\ 
		        Butterfly & 4 & $(x,xy+y^5+y^7)$ \\ 
		        Gulls & 4 &  $(x,xy^2+y^4+y^5)$  \\ 
		        Sharksfin & 4 & $(x^2+y^3,x^3+y^2)$ \\ \hline
        	\end{tabular}	\vspace{8pt}
        	\caption{Map-germs $\C^2\to\C^2$ of $\A$-codimension $\leq 4$.} 
	\label{table_2to2}
\end{table}

\begin{table}
		\centering
		    \begin{tabular}{|c|c|l|}\hline 
		        Type &  $\A$-codim & Normal form \\ \hline \hline
		        Immersion  & 0 & $(x,y,0)$ \\ 
		        $S_0$  & 2 & $(x,y^2,xy)$ \\ \hline
		        $S_k$ & $k+2$ & $(x,y^2,y^3+x^{k+1}y)$\\ 
		        $B_k$ & $k+2$ & $(x,y^2,x^2y+y^{2k+1})$ \\ 
		        $H_k$ & $k+2$ & $(x,y^3,xy+y^{3k-1})$ \\ 
		        $C_3$ & 5 &  $(x,y^2,xy^3+x^3y)$  \\ 
		        $P_3$ & 5 & $(x,xy+y^3,xy^2+cy^4)$ 
		        \\ \hline
        	\end{tabular} \vspace{8pt}
        	\caption{Map-germs $\C^2\to\C^3$ of $\A$-codimension $\leq 5$ 
	($S_1=B_1=H_1$, $k\le 3$). 
	$P_3$ is unimodal with a parameter $c$.} \label{table_2to3}
\end{table}

\begin{table}
		\centering
	        \begin{tabular}{|c|c|l|}\hline 
		        Type &  $\A$-codim & Normal form \\ \hline \hline
		        $A_0$  & 0 & $(x,y,z)$ \\ 
		        $A_1$  & 1 & $(x,y,z^2)$ \\ 
		        $A_2$ & 2 & $(x,y,yz+z^3)$\\ 
		        $A_3$ & 3 & $(x,y,yz+xz^2+z^4)$\\ \hline 
		        $A_4$ & 4 & $(x,y,yz+xz^2+z^5)$\\ 
		        $C$ & 4 & $(x,y,z^3+(y^2+x^2)z)$\\ 
		        $D$ & 4 & $(x,y,yz+z^4+x^2z^2)$\\ 
		        $I_{2,2}$ & 4 &  $(x,yz,y^2+z^2+xy)$\\   \hline
	        \end{tabular} \vspace{8pt}
	    	\caption{Map-germs $\C^3\to\C^3$ of $\A$-codimension $\leq$ 4.} \label{table_3to3}
\end{table}

Our main result is to determine the precise form of $Tp_\A$ for 
$\A$-singularities in these classification. 

\begin{thm}\label{main_thm} 
The Thom polynomials $Tp_\A(\eta)$ for $\A$-singularities 
 in Tables \ref{table_2to2},  \ref{table_2to3} and  \ref{table_3to3} 
are given in Tables \ref{table_tp2to2},  \ref{table_tp2to3} and  \ref{table_tp3to3}, respectively. 
\end{thm}

\begin{rem}{\rm 
Note that 
for some $\A$-types $\eta$,  
$Tp_{\A}(\eta)$ are written in quotient Chern classes 
(e.g., Swallowtail, Butterfly and Sharksfin in Table \ref{table_2to2});  
$Tp_\A(\eta)=Tp(\eta)$. 
It is the case that the $\A$-orbit $\eta$ is an open dense subset 
of its $\K$-orbit. 
If $\eta$ is a proper subset in its $\K$-orbit, 
then 
$Tp_{\A}(\eta)$ can not be written in quotient Chern classes. 
}
\end{rem}

\begin{table}
\centering
\begin{tabular}{|c|l|}\hline
Type  & Thom polynomial\\ \hline\hline
Fold (=$A_1$)  & $\bar{c}_1$ \\ \hline
Cusp (=$A_2$)  & $\bar{c}_1^2+\bar{c}_2$ \\ \hline
Swallowtail (=$A_3$)   &$\bar{c}_1^3+3\bar{c}_1\bar{c}_2+2\bar{c}_3$ \\ \hline
Lips/Beaks    & $-2c_1^3+5c_1^2c'_1-4c_1{c'_1}^2-c_1c_2+c_2c'_1+{c'_1}^3$\\ \hline
Goose  & \makecell[l]{$2c_1^4+5c_1^2c_2+4c_2^2-7c_1^3c'_1-10c_1c_2c'_1$\\$+9c_1^2{c'_1}^2+5c_2{c'_1}^2-5c_1{c'_1}^3+{c'_1}^4-2c_1^2c'_2$\\$-6c_2c'_2+4c_1c'_1c'_2-2{c'_1}^2c'_2+2{c'_2}^2$}\\ \hline
Gulls  & \makecell[l]{$6c_1^4-c_1^2c_2-4c_2^2-17c_1^3c'_1+4c_1c_2c'_1$\\$+17c_1^2{c'_1}^2-3c_2{c'_1}^2-7c_1{c'_1}^3+{c'_1}^4+2c_1^2c'_2$\\$+6c_2c'_2-4c_1c'_1c'_2+2{c'_1}^2c'_2-2{c'_2}^2$}\\ \hline
Butterfly (=$A_4$) & $\bar{c}_1^4+6\bar{c}_1^2\bar{c}_2+2\bar{c}_2^2+9\bar{c}_1\bar{c}_3+6\bar{c}_4$\\ \hline
Sharksfin (=$I_{2,2}$)& $\bar{c}_2^2-\bar{c}_1\bar{c}_3$ \\ \hline
\end{tabular} \vspace{8pt}
\caption{Thom polynomials for $\C^2\to\C^2$.}\label{table_tp2to2}
\end{table}
	
\begin{table}
\centering
\begin{tabular}{|c|l|}\hline
Type   & Thom polynomial\\ \hline\hline
$S_0$
& $\bar{c}_2$ \\ \hline
$B_1(=S_1)$    & 
$-3 c_1^3 + 4 c_1 c_2 + 4 c_1^2 c'_1 - 2 c_2 c'_1 
- c_1 {c'_1}^2 - 3 c_1 c'_2 + c'_1 c'_2 + c'_3$ \\ 
\hline
$S_2$ & \makecell[l]{
$13 c_1^4 - 22 c_1^2 c_2 + 3 c_2^2 - 21 c_1^3 c'_1 + 19 c_1 c_2 c'_1 $\\
$+  9 c_1^2 {c'_1}^2 - 3 c_2 {c'_1}^2 - {c_1} {c'_1}^3 + 14 c_1^2 c'_2 - 4 c_2 c'_2 $
\\$- 9 c_1 c'_1 c'_2 +{c'_1}^2 c'_2 + {c'_2}^2 - 6 c_1 c'_3 + 2 c'_1 c'_3$} \\ 
\hline
$B_2$  & \makecell[l]{
$11 c_1^4 - 22 c_1^2 c_2 + c_2^2 - 17 c_1^3 c'_1 + 21 c_1 c_2 c'_1$\\
$ + 7 c_1^2 {c'_1}^2 - 5 c_2 {c'_1}^2- c_1 {c'_1}^3 + 10 c_1^2 c'_2$\\
$ - 5 c_1 c'_1 c'_2 + {c'_1}^2 c'_2 - {c'_2}^2 - 10 c_1 c'_3 + 4 c'_1 c'_3$} \\ 
\hline
$H_2$
& $\bar{c}_2^2+\bar{c}_1\bar{c}_3+2\bar{c}_4$\\ \hline
$S_3$ & \makecell[l]{
$-71 c_1^5 + 149 c_1^3 c_2 -  48 c_1 c_2^2 + 132 c_1^4 c'_1 - 174 c_1^2 c_2 c'_1$\\
$ + 20 c_2^2 c'_1 - 76 c_1^3 {c'_1}^2 + 53 c_1 c_2 {c'_1}^2 + 16 c_1^2 {c'_1}^3$\\
$ - 4 c_2 {c'_1}^3 -  c_1 {c'_1}^4 - 82 c_1^3 c'_2 + 53 c_1 c_2 c'_2 + 75 c_1^2 c'_1 c'_2 $\\
$- 17 c_2 c'_1 c'_2 - 18 c_1 {c'_1}^2 c'_2 + {c'_1}^3 c'_2 - 11 c_1 {c'_2}^2 + 3 c'_1 {c'_2}^2$\\
$ + 39 c_1^2 c'_3 -  9 c_2 c'_3 - 24 c_1 c'_1 c'_3 + 3 {c'_1}^2 c'_3 + 3 c'_2 c'_3$}\\ \hline
$B_3$  & \makecell[l]{
$-110 c_1^5 + 286 c_1^3 c_2 - 76 c_1 c_2^2 + 192 c_1^4 c'_1 - 356 c_1^2 c_2 c'_1$\\
$ + 32 c_2^2 c'_1 - 104 c_1^3 {c'_1}^2 + 134 c_1 c_2 {c'_1}^2 + 24 c_1^2 {c'_1}^3 - 16 c_2 {c'_1}^3$\\
$ - 2 c_1 {c'_1}^4 - 100 c_1^3 c'_2 + 54 c_1 c_2 c'_2 + 70 c_1^2 c'_1 c'_2 - 18 c_2 c'_1 c'_2$\\
$ - 20 c_1 {c'_1}^2 c'_2 + 2 {c'_1}^3 c'_2 + 10 c_1 {c'_2}^2 - 2 c'_1 {c'_2}^2 + 106 c_1^2 c'_3 $\\
$- 6 c_2 c'_3 - 72 c_1 c'_1 c'_3 + 14 {c'_1}^2 c'_3 - 6 c'_2 c'_3$}\\ 
\hline
$H_3$  & \makecell[l]{
$-48 c_1^5 + 156 c_1^3 c_2 - 90 c_1 c_2^2 + 80 c_1^4 c'_1 - 182 c_1^2 c_2 c'_1$\\
$ + 42 c_2^2 c'_1 -  36 c_1^3 {c'_1}^2 + 48 c_1 c_2 {c'_1}^2 + 4 c_1^2 {c'_1}^3$\\
$ - 2 c_2 {c'_1}^3 -  60 c_1^3 c'_2 + 84 c_1 c_2 c'_2 + 46 c_1^2 c'_1 c'_2 - 26 c_2 c'_1 c'_2$\\
$ -  6 c_1 {c'_1}^2 c'_2 - 12 c_1 {c'_2}^2 + 2 c'_1 {c'_2}^2 + 45 c_1^2 c'_3 - 27 c_2 c'_3$\\
$ -  27 c_1 c'_1 c'_3 + 2 {c'_1}^2 c'_3 + 9 c'_2 c'_3$}\\ \hline
$C_3$  & \makecell[l]{
$-33 c_1^5 + 66 c_1^3 c_2 - 3 c_1 c_2^2 +62 c_1^4 c'_1 - 85 c_1^2 c_2 c'_1$\\
$  + c_2^2 c'_1 - 38 c_1^3 {c'_1}^2 +36 c_1 c_2 {c'_1}^2 + 10 c_1^2 {c'_1}^3 - 5 c_2 {c'_1}^3 $\\
$    - c_1 {c'_1}^4 - 30 c_1^3 c'_2  +25 c_1^2 c'_1 c'_2 - 8 c_1 {c'_1}^2 c'_2 $\\
$    + {c'_1}^3 c'_2 + 3 c_1 {c'_2}^2 - c'_1 {c'_2}^2  +30 c_1^2 c'_3 - 22 c_1 c'_1 c'_3 + 4 {c'_1}^2 c'_3$}\\ 
\hline
$P_3$  & \makecell[l]{
$-16 c_1^5 + 48 c_1^3 c_2 -  24 c_1 c_2^2 + 28 c_1^4 c'_1 - 58 c_1^2 c_2 c'_1$\\
$ + 11 c_2^2 c'_1  -14 c_1^3 {c'_1}^2 + 17 c_1 c_2 {c'_1}^2$\\
$ + 2 c_1^2 {c'_1}^3 - c_2 {c'_1}^3 - 20 c_1^3 c'_2 + 24 c_1 c_2 c'_2 + 17 c_1^2 c'_1 c'_2$\\
$ - 8 c_2 c'_1 c'_2-  3 c_1 {c'_1}^2 c'_2 - 4 c_1 {c'_2}^2 + c'_1 {c'_2}^2 + 14 c_1^2 c'_3$\\
$ - 6 c_2 c'_3 -  9 c_1 c'_1 c'_3 + {c'_1}^2 c'_3 + 2 c'_2 c'_3$}\\ 
\hline
\end{tabular} \vspace{8pt}
\caption{Thom polynomials for $\C^2\to\C^3$.}\label{table_tp2to3}
\end{table}

\begin{table}
\centering
\begin{tabular}{|c|l|}\hline
  Type & Thom polynomial\\ \hline\hline
$C$  & \makecell[l]{
$2 c_1^4 + c_1^2 c_2 - 2 c_2^2 + 3 c_1 c_3 - 7 c_1^3 c'_1 - 3 c_3 c'_1$\\
$ +  9 c_1^2 {c'_1}^2 - c_2 {c'_1}^2 - 5 c_1 {c'_1}^3 + {c'_1}^4 - 2c_1^2 c'_2$\\
$ + 4 c_2 c'_2 +  2 c_1 c'_1 c'_2 - 2 {c'_2}^2 - 2 c_1 c'_3 + 2 c'_1 c'_3$}\\ \hline
$D$   & \makecell[l]{
$18 c_1^4 - 21 c_1^2 c_2 - 2 c_2^2 + 8 c_1 c_3 - 45 c_1^3 c'_1 + 31 c_1 c_2 c'_1$\\
$ - 6 c_3 c'_1 + 40 c_1^2 {c'_1}^2 - 12 c_2 {c'_1}^2 - 15 c_1{ c'_1}^3 + 2 {c'_1}^4 +13 c_1^2 c'_2$\\
$ + 4 c_2 c'_2 - 17 c_1 c'_1 c'_2 + 6 {c'_1}^2 c'_2 - 2 {c'_2}^2 - 8 c_1 c'_3 +  6 c'_1 c'_3$}\\ \hline
\end{tabular} \vspace{8pt}
\caption{Thom polynomials for $\C^3\to\C^3$. 
 Other types, i.e. $A_\mu$ and $I_{2,2}$, have the same Thom polynomials of 
 stable singularities types as appeared in Table \ref{table_tp2to2}.}
\label{table_tp3to3}
\end{table}

To compute $Tp_{\A}(\eta)$, 
we make use of an effective method due to R. Rim\'anyi, called 
the \textit{restriction method} \cite{Rimanyi, FR02, Kaz03}. 
Below we demonstrate the computation of $Tp_{\A}(B_1)$ 
for the $\A$-singularity 
$$B_1(=S_1)\; (x, y^2, y^3+x^2y): (\C^2,0)\to(\C^3,0).$$
 It has $\A$-codimension $3$, thus we may write 
\begin{align}
Tp_{\A}(B_1)=&\;\;  x_1c_1^3+x_2c_1c_2+x_3c_1^2c'_1+x_4c_1{c'_1}^2+x_5c_2{c'_1}\nonumber\\
& +x_6{c'_1}^3+x_7c_1c'_2+x_8c'_1c'_2+x_9c'_3 \nonumber
\end{align}
with unknown coefficients $x_1,\ldots,x_9$. 
So our task is to determine these coefficients. 
The key point is a simple fact that weighted homogeneous map-germs admit 
a canonical action of the torus $T=\C^*=\C-\{0\}$. 
Take an $\A_e$-versal unfolding of type $B_1$ 
$$
B_1:\C^2\times \C\rightarrow\C^3,\ (x,y,a)\mapsto (x,y^2,y^3+x^2y+ay),
$$
then the $T$-actions on the source and the target are diagonal: 
$$
\rho_0=\alpha\oplus\alpha\oplus\alpha^{\otimes 2},\ \rho_1= \alpha\oplus\alpha^{\otimes 2}\oplus\alpha^{\otimes 3}\quad (\alpha\in T). 
$$

Take the dual tautological line bundle $\gamma=\Ost_{\Proj^N}(1)$ over 
a projective space of high dimension $N \gg 0$  or say, $\Proj^\infty=B\C^*$, 
the classifying space for $\C^*$. 
We set the following vector bundles:
$$
E_0:=\gamma\oplus\gamma, \;\; E_0':=\gamma^{\otimes 2},
\;\; E_1:=\gamma\oplus \gamma\oplus^{\otimes 2} \gamma^{\otimes 3},
$$
and then by using the normal form of $B_1$ and $\C^*$-action 
we can define a holomorphic family of map-germs so that 
the restriction of each fiber is $\A$-equivalent to 
the versal unfolding of type $B_1$:   
$$
\begin{tikzcd}
E_0\oplus E'_0 \arrow{rr}{f_{B_1}}\arrow{rd} & & E_1\arrow{ld}\\
& \Proj^\infty &
\end{tikzcd}
$$

Put $a=c_1(\gamma)$ and then $H^*(\Proj^\infty)=\Z[a]$, 
which is canonically isomorphic to the cohomology of the total spaces of these vector bundles 
via the pullback of the projection to the base $\Proj^\infty$. 
Chern classes of those bundles are written by
$$
c(E_0)=(1+a)^2,\quad c(E'_0)=1+2a,
$$

$$c(E_1)=(1+a)(1+2a)(1+3a),
$$
therefore $c_1(E_0)=2a$, $c_2(E_0)=a^2$, $c'_1(E_1)=6a$, $c'_2(E_1)=11a^2$, $c'_3(E_1)=6a^3$, 
and the top Chern class of $E_0\oplus E'_0$ is equal to $2a^3$. 

We substitute them for Chern classes in the above form of $Tp_\A(B_1)$. 
Then it equals the top Chern class of $E_0\oplus E'_0$ because 
the $B_1$-singularity locus of the universal map $f_{B_1}$ 
is just the zero section of $E_0\oplus E'_0$. 
Hence, we get the following equation 
(called the {\it principal equation} in \cite{Rimanyi, FR}): 
$$
8 x_1 + 2 x_2 + 24 x_3 + 72 x_4 + 6 x_5 + 216 x_6 + 22 x_7 + 66 x_8 + 
 6 x_9 = 2.
$$
Similarly, to each of other singularity types 
with $\A$-codimension less than or equal to that of $B_1$, 
i.e., $A_0$ (immersion),  $S_0$, 
we can associate a holomorphic family, the universal map of that type. 
Apply $Tp_{\A}(B_1)$ to the holomorphic family, 
then the result must be zero, 
because the $B_1$-singularity does not appear. 
Consequently, we obtain ten more equations as follows 
(called the {\it homogeneous equations}): 
\begin{align}
&g_{01} : x_6=0 \nonumber\\
&g_{02} : x_1 + x_3 + x_4 + x_6=0 \nonumber\\
&g_{03} : x_4 + 3 x_6 + x_8=0 \nonumber\\
&g_{04} : x_3 + 2 x_4 + 3 x_6 + x_7 + x_8=0\nonumber \\
&g_{05} : 3 x_1 + x_2 + 3 x_3 + 3 x_4 + x_5 + 3 x_6 + x_7 + x_8=0 \nonumber\\
&g_{06} : 2 x_3 + 4 x_4 + x_5 + 6 x_6 + 2 x_7 + 3 x_8 + x_9=0 \nonumber\\
&g_{07} : x_1 + 2 x_3 + 4 x_4 + 8 x_6 + x_7 + 2 x_8=0\nonumber \\
&g_{08} : x_1 + 3 x_3 + 9 x_4 + 27 x_6 + 2 x_7 + 6 x_8=0 \nonumber\\
&g_{09} : 3 x_1 + x_2 + 7 x_3 + 16 x_4 + 2 x_5 + 36 x_6 + 6 x_7 + 13 x_8 + 2 x_9=0\nonumber \\
&g_{10} : 3 x_1 + x_2 + 8 x_3 + 21 x_4 + 3 x_5 + 54 x_6 + 7 x_7 + 19 x_8 + 2 x_9=0\nonumber
\end{align}
The first six equations come from $A_0: (x,y) \mapsto (x,y,0)$: 
the symmetry group is $T^3$ acting on the source and the target in an obvious way, 
that induces $f_{A_0}: E_0 \to E_1$ ($E_0'=\emptyset$). 
The last four come from $S_0$. 

Solving the above system of eleven linear equations of unknowns $x_1, \cdots, x_9$ 
(in fact, the equation is over-determined), we obtain 
$$
Tp(B_1)=-3 c_1^3 + 4 c_1 c_2 + 4 c_1^2 c'_1 - c_1 {c'_1}^2  - 2 c_2 c'_1 - 3 c_1 c'_2 + c'_1 c'_2 + c'_3.
$$

In entirely the same way, 
we can compute Thom polynomials for other $\A$-singularities types.

\begin{rem}\label{quasi_homog}{\rm 
Necessary inputs for computing Thom polynomials via the restriction method 
are the weights and degrees of quasi-homogeneous normal forms 
of germs (or jets)  in the classification. 
In Tables \ref{table_2to3} and \ref{table_3to3}, 
all the normal forms are quasi-homogeneous. 
However, 
in Table \ref{table_2to2}, 
Butterfly, Gulls  and Sharksfin are not so. 
Then, we instead work over $5$-jets;
there is an invariant stratification of $J^5(2,2)$ up to codimension $4$ 
so that those strata correspond to 
 \textit{equisingularity types} (or topological $\A$-classification) 
 of plane-to-plane germs  \cite{Rieger}. 
Every stratum admits a representative of quasi-homogeneous $5$-jet, 
thus the restriction method works and one gets corresponding Thom polynomials 
as in Table \ref{table_tp2to2}. 
A trouble comes to $\A$-types of codimension $5$; 
Butterfly $(x, xy+y^5+y^7)$ and Elder butterfly $(x, xy+y^5)$ 
are different $\A$-types but equisingular; they are $7$-$\A$-determined 
of  $\A$-codimension  $4$ and $5$, respectively. 
It is very hard to compute the $Tp_\A$ of Elder butterfly, 
because Butterfly in $7$-jets is not quasi-homogeneous 
(so the restriction method does not work). 
To compute it, we need some different technique, e.g. 
a resolution of the $\A$-orbit of Elder butterfly in $J^7(2,2)$. 
In fact, 
Tables \ref{table_tp2to2}, \ref{table_tp2to3} and \ref{table_tp3to3} 
are the best results using the restriction method, 
since 
higher codimensional $\A$-types beyond  tables have 
normal forms with moduli parameters or quasi-homogeneous normal forms. 
 }
 \end{rem}

\section{Counting singular projections}

\subsection{Singularities of projections}
We are concerned with counting lines $l$ tangent to a smooth projective variety 
$X^m \subset \Proj^{n+1}$ at 
some point $x \in X$ with prescribed contact types. 
First,  take a point $q \, (\not=x) \in l$ and consider 
the linear projection  
from $q$ mapped to a hyperplane transverse to $l$ at $x$; 
that defines a holomorphic family 
over the space $\F$ of triples $(x, l, q)$ 
with $x \in X$ and $x, q (\not=x) \in l$: 
$$\varphi_{(x,l,q)}: 
(\C^m,0) = (X, x) \hookrightarrow (\Proj^{n+1}-\{q\}, x) \to (\Proj^n, x) = (\C^n,0).$$
The $\A$-type of $\varphi_{(x,l,q)}$ 
does not depend on the choice of general $q$ lying on $l$, 
while it can be more degenerate when taking $q$ at some special position.  
So we call  the map-germ from general $q$ 
 the {\it projection of $X$ along $l$ at $x$}, for short. 
Below, we summarize some classification results 
(Theorem \ref{proj23}, Theorem \ref{proj24}), 
that tell us of which $\A$-types should be considered in our problem. 
For such an $\A$-type $\eta$ of $(\C^m,0) \to (\C^n,0)$, 
we set the {\it $\eta$-singularity locus of $X$} to be the locus of points $x \in X$ 
having a tangent line along which  the projection is of type $\eta$. 
What we want is to compute the degree of the $\eta$-singularity locus. 
Here we use our universal polynomial $Tp_\A(\eta)$. 
Note that the Thom polynomial can not directly be applied to $\varphi_{(x,l,q)}$, 
because the base space $\F$ is not compact.  
A key idea is to introduce a certain family of $\varphi$ 
over the flag manifold of $(x, l)$ by choosing $q=q_{(x,l)}$ 
at the `infinity on $l$ with respect to $x$' 
so that generically $q$ does not meet any special positions,  
see (\ref{q}) and (\ref{varphi}) in \S 3.2.

\

\t
$\bullet \;\;$ {\bf Surfaces in $3$-space.} 
The contact of any lines with a {\it generic} smooth surface in $\Proj^3$ 
has been studied in \cite{Arnold, Platonova} (cf.   \cite{Landis, Kabata}), 
and for a singular surface with crosscaps,  in \cite{YKO, West}. 
The flag manifold of $(x, l)$ in this case has dimension $4$. 
Thus, by a standard argument of transversality, 
the projection along lines of a generic surface 
must have singularities of $\A$-codimension at most $4$. 
However, it turns out that not all $\A$-types may appear,  
for some constraint is caused by the geometric setting. 
Indeed,  it is shown that 

\begin{thm}{\rm \cite{Platonova, Kabata, YKO}} \label{proj23} 
For a generic surface in $\Proj^3$ with ordinary singularities, it holds that  
for any point $x$ of the surface and 
for any line $l$ passing through $x$, 
the projection of the surface along $l$ at $x$ is $\A$-equivalent to 
one of all but the Goose singularity in Table  \ref{table_2to2}. 
\end{thm}

\begin{rem}\label{Platonova}
{\rm 
The main point of the proof is briefly explained as follows. 
In  \cite{Platonova} (also \cite{Kabata, SKSO, YKO}), 
one considers the holomorphic family $\varphi_{(x,l,q)}$ 
of central projections 
over the space $\F$;  
the transversality condition of its jet extension with $\A$-orbits is imposed 
for classifying singularities of the projection, 
that is, this condition provides  
the required genericity of embeddings of the surface 
by a version of transversality (or Bertini-type) theorem. 
A major task is to find the local defining equation of 
the $\eta$-singularity locus for $\varphi_{(x,l,q)}$
in $\F$ 
-- in particular, for the Goose type, 
it contains an equation of $q$ (with coefficients depending on $(x,l)$), 
which defines the special position of viewpoints on $l$, 
while there appears  
no equations of $q$ for other types $\eta$ in Table  \ref{table_2to2}. 
That is why this type is excluded in Theorem \ref{proj23}.   
}
\end{rem}

Let us see some geometric characterizations. Let $X$ be a surface in $\Proj^3$ and $x \in X$. 
In general, there are exactly two asymptotic lines, 
i.e. tangent lines at $x$ having at least $3$-point contact with $X$; 
the projection along a non-asymptotic line at $x$ is of type Fold, 
and the projection along an asymptotic line is generically of Cusp. 
A {\it flecnodal point} is the point so that  
one of asymptotic line has at least $4$-point contact 
(then $X \cap T_xX$ has a node at $x$ with an inflection); 
in other words, the projection along the line is of  Swallowtail or worse. 
A {\it parabolic point} is the point which has a unique asymptotic line (doubly counted); 
the projection along the asymptotic line is of Lips/Beaks or worse. 
At every parabolic point 
the Goose singularity is generically observed in the projection 
from a special viewpoint $q$ on the asymptotic line,  
just as mentioned in Remark \ref{Platonova}. 
The {\it parabolic curve} (resp. {\it flecnodal curve}) is defined to be 
the closure of the locus of parabolic points (resp.  flecnodal points). 
The parabolic curve becomes to be smooth,  
and the flecnodal curve is smooth but only finitely many nodes; 
furthermore, these two curves meet each other tangentially at {\it cusps of Gauss} -- 
those points are actually the locus of Gulls singularity, 
i.e. the projection along the asymptotic line is of type Gulls. 
The flecnodal curve may be tangent to asymptotic lines at some isolated points on the curve, 
that corresponds to the locus of Butterfly. 

In case that $X$ admits ordinary singularities, 
the resolution of singularities gives a stable map $f: M \to \Proj^3$ (with $X=f(M)$ and $M$ smooth). 
Notice that the corank $2$ singularity $I_{2,2}$ appear 
in the projection at a crosscap $x \in X$; more precisely, 
it appears in the composed map of $f$ and the projection along 
$l=df_p(T_pM)$ ($x=f(p)$) (cf. \cite{West}). 
Projective differential geometry of parabolic and flenodal curves around a crosscap 
has been explored rather recently in \cite{YKO}; 
it is shown that the preimage in $M$ of each curve has a node at the crosscap point, 
and each smooth branch of the flecnodal curve has $4$-point contact 
with a smooth branch of the parabolic curve. 

\

\t
$\bullet \;\;$ {\bf Surfaces in $4$-space.} 
The contact of any lines with  
a {\it generic} smooth surface in $\Proj^4$ 
has firstly been considered in \cite{Mond82} 
(see also \cite[Appendix]{DK}). 
The projection along lines must have $\A$-codimension $\le 5$, 
but some $\A$-types do not appear. 

\begin{thm}{\rm \cite{Mond82, DK}} \label{proj24} 
For a generic surface in $\Proj^4$, it holds that  
for any point $x$ of the surface and for any line $l$ passing through $x$, 
the projection of the surface along $l$ at $x$ is $\A$-equivalent to 
one of types in Table \ref{table_2to3} 
except for $S_2, S_3, B_3$ and $C_3$. 
\end{thm}

For a smooth surface $X$ in $\Proj^4$, 
the $2$-jet at a point $x \in X$ may be expressed by $(u,v,Q_1, Q_2)$ 
in an affine chart centered at $x$, 
where $Q_1, Q_2$ are binary forms in local coordinates $u, v$ of $X$. 
If $Q_1$ and $Q_2$ are dependent, 
the point $x$ is called an {\it inflection point}. 
If not, they define a line $L$ in the space of binary forms $\simeq \Proj^2$. 
Let $C \subset \Proj^2$ be the curve of perfect squares, a conic; 
a point of $L \cap C$ is the square of a linear form,  and 
the kernel of the linear form is called an {\it asymptotic line} of $X$ at $x$. 
If $L$ is transverse  (resp. tangent) to $C$, 
the point $x$ is called {\it hyperbolic} (resp. {\it parabolic}). 
Let $X$ be generic as in Theorem \ref{proj24}. 
At a hyperbolic point, 
the projection along an asymptotic line is of type $B_1(=S_1)$ or $B_2$;  
the $B_2$-singularity occurs on a curve on $X$,  
while other types $S_2, S_3, B_3$ and $C_3$ are observed 
from some special viewpoints on an asymptotic line, 
thus excluded from our consideration (see \cite[Appendix]{DK}). 
At a parabolic point, 
the projection along the unique asymptotic line is of type $H_2$ or worse $H_3, P_3$;  
in particular, 
the $H_2$-locus is nothing but the parabolic curve. 
The curve admits some discrete points where $H_3$ and $P_3$-singularities appear; 
at the latter points, the parabolic curve meets the $B_2$-curve tangentially 
(the asymptotic line is also tangent to the curve). 
Inflection points are singular points of the parabolic curve; 
the projection along any tangent line is of type $B_1$ or worse. 
As a remark, the classification of singularities of orthogonal projections  in 
the affine setting has some difference \cite{NT}, e.g., 
the $S_2$-singularity occurs on a curve  
and each of types $S_3, B_3$ and  $C_3$ appears on some discrete points.

\

\t
$\bullet \;\;$ {\bf Primals in $4$-space.}  
In the same way as above, 
the contact of any lines with 
a generic smooth $3$-fold $X$ in $\Proj^4$ 
(classically called a {\it primal}) 
has partly been studied in Nabarro \cite{Anna} 
as an application of the $\A$-classification of $(\C^3,0) \to (\C^3,0)$; 
several loci in $X$ up to codimension one are characterized 
by singular projections of $\A$-codimension $\le 4$ (Table \ref{table_3to3}).  
As for the geometry of higher codimensional loci in $X$, 
it seems that nothing has been known so far.

\subsection{Surfaces in $\Proj^3$} 

Let $X \subset \Proj^3$ be a reduced surface of degree $d$ having 
only ordinary singularities, that is,  
crosscap and double and triple points 
-- they are locally defined, respectively, by equations 
$$xy^2-z^2=0, \quad xy=0, \quad xyz=0.$$
Take a resolution of singularities, 
we have a stable map $f: M \to \Proj^3$ with $X=f(M)$ and $M$ smooth 
(it is not necessary that $M$ is embedded in some higher dimensional projective space 
and  $f$ is a linear projection). 
We denote the number of crosscaps and triple points by $C$ and $T$ respectively. 
Let $D\; (\subset X)$ denote the double point locus of $X$; 
then $D$ is a space curve with triple points and singular points at crosscaps. 
Denote by $\epsilon_0$ the degree of $D$. 

We denote the hyperplane class of $\Proj^3$ by 
$$a=c_1(\Ost_{\Proj^3}(1))\in H^2(\Proj^3).$$  
Put integers $\xi_1, \xi_2, \xi_{01}$ such that 
$$
f_*(1)= d  a, \;\;
f_*c_1(TM)=\xi_1  a^2, \;\;
f_*(c_1(TM)^2)=\xi_2 a^3, $$
$$
f_* c_2(TM)=\xi_{01} a^3\in H^*(\Proj^3). 
$$

Now we recall in Table \ref{tp_stable2to3} 
the double and triple point formulae 
of maps $f: M \to N$ from a surface into $3$-space 
 (i.e. Thom polynomials for $A_0^2$ and $A_0^3$), 
where $s_0=f^*f_*(1) \in H^2(M)$ and $s_1=f^*f_*(\bar{c}_1) \in H^4(M)$ 
with $\bar{c}_i=c_i(f^*TN-TM)$  (see e.g. \cite{Kleiman, Kaz03}): 
\begin{table}[htb]
		\centering
    	\begin{tabular}{|c|l|}\hline
		    Type  & Thom polynomial\\ \hline\hline
		    $A_0^2$   & $s_0-\bar{c}_1$ \\ \hline
		    $A_1$   & $\bar{c}_2$ \\ \hline
		    $A_0^3$   & $\frac{1}{2}({s_0}^2-s_1-2s_0\bar{c}_1+2{\bar{c}_1}^2+2\bar{c}_2)$ \\ \hline
    	\end{tabular} \vspace{8pt}
    	\caption{$Tp$ of stable multi-singularities of $\C^2 \to \C^3$.}
	\label{tp_stable2to3}
\end{table}

Applying these universal polynomials to $f: M \to \Proj^3$, 
it immediately follows that  
\begin{align*}
C& =\int_{\Proj^3} f_*(Tp(A_1)(f))=6d-4\xi_1+\xi_2-\xi_{01}, \\
T& =\int_{\Proj^3}\frac{1}{3}(f_*(Tp(A_0^3)(f)))\\
& =\frac{1}{6}(44d-12d^2+d^3-24\xi_1+3d\xi_1+4\xi_2-2\xi_{01}), \\
\epsilon_0& 
=\int_{\Proj^3}\dfrac{1}{2}f_*(Tp(A_0^2)(f))=\dfrac{1}{2}(d^2-4d+\xi_1). 
\end{align*}
Thus we may express $\xi_1, \xi_2, \xi_{01}$ in terms of $\epsilon_0, C, T$ as follows -- 
these formulae are actually classical  \cite[Prop.1]{Piene}; 
indeed, $\xi_{01}-4$ and $\xi_2+1$ coincide with  
the Zeuthen-Segre invariant and the Castelnuovo-Enriques invariant of $X$, respectively, 
in   \cite[pp.223--224]{SR}: 

\begin{lem}\label{Enriques} 
It holds that 
\begin{align*}
	\xi_1&=d(4-d)+2\epsilon_0,\\
	\xi_2 &= d(d-4)^2 +(16-3d)\epsilon_0 +3T -C,\\
	\xi_{01} &= d(d^2-4d+6)+(8-3d)\epsilon_0 +3T -2C.
\end{align*}
\end{lem}

Now we consider projections of $X$ along lines. 
Let $G_3$ denote the Grassmanian manifold 
of $2$-dimensional subspaces $\lambda \subset \C^4$, 
or equivalently the Grassmanian of all lines $l \subset \Proj^3$; 
$$G_3:=G(2,4)=G(\Proj^1,\Proj^3).$$
Note that $\dim G_3=4$. 
Take the flag manifold 
$$F_3=\{(x,l)\in \Proj^3\times G_3\ |\ x\in l\},$$ 
and  the pullback 
$$F_M=\{(p,l)\in M\times G_3\ |\ f(p)\in l\}$$ 
together with the following diagram (cf. \cite{Kulikov}): 
$$
\begin{tikzcd}
		& F_3\arrow{ld}[swap]{\pi}\arrow{rd}{\bar{h}} &  \\
		\Proj^3 & F_M \arrow{u}{i}\arrow{ld}{g}\arrow{r}[swap]{h} & G_3   \\
		M \arrow{u}{f} & &
\end{tikzcd}
$$
Recall the Euler sequence over $\Proj^3$: 
\begin{equation}
\label{exactsq:P3}
\begin{tikzcd}
0 \arrow{r}&\gamma\arrow{r}&\epsilon^4\arrow{r}&Q\arrow{r}& 0
\end{tikzcd}	
\end{equation}
where 
$\epsilon^4$ is the trivial vector bundle of rank $4$ over $\Proj^3$, 
$\gamma$ is the tautological line bundle $\Ost_{\Proj^3}(-1)$ 
and $Q$ is the quotient subbundle of rank 3. 
The projectivization of $Q$ 
is denoted by $\pi:P(Q)\to\Proj^3$; 
the fiber at $x$ is the projective space $P(Q_x)$.
A point of  $P(Q_x)$ corresponds to a line of $Q_x$, 
that defines a projective line $l \subset \Proj^3$ passing through $x$, 
i.e. a flag $x \in l$. 
That leads to the canonical identification 
$$P(Q) = F_3.$$

Over the total space $P(Q)$,  
the tautological line bundle $L$ is defined:  
$$
 L \;  := \Ost_{P(Q)}(-1) \,= \{(x, l, v)\in \pi^*Q\,|\, [v] = l \; \mbox{or}\; v=0  \}. 
$$
This leads to an exact sequence of vector bundles over $P(Q)$: 
\begin{equation}
\label{exactsq:P(Q)}
	\begin{tikzcd}
	0\arrow{r}& L\arrow{r} & \pi^*Q\arrow{r} & Q' \arrow{r} & 0.
	\end{tikzcd}
\end{equation}
The Euler sequence over the Grassmaniann $G_3$ is 
\begin{equation}
\label{exactsq:G3}
\begin{tikzcd}
0 \arrow{r} & S\arrow{r}&\epsilon^4_{G_3}\arrow{r}&V\arrow{r}&0 
\end{tikzcd}
\end{equation}
where $S$ is the tautological bundle of rank $2$: 
$$S=\{(\lambda,v)\in G_3\times\C^4\ |\ v\in\lambda\},$$
and $V$ is the quotient bundle of rank $2$. 
Take the pullback of this sequence via $\bar{h}: P(Q)=F_3 \to G_3$. 
At a point $(x, l) \in P(Q)$, 
the fiber $\bar{h}^*S_{(x, l)}$ is identified with 
the $2$-dimensional subspace $\lambda$ 
so that $P(\lambda)=l \subset \Proj^3$ 
(the relation $x \in l$ corresponds to $\gamma_x \subset \lambda$, 
where $\gamma_x$ is the line representing $x$ in $\C^4$). 
So we may choose a line 
$L_l \subset \C^4$ 
so that 
$$\lambda=\lambda_{(x, l)} = \gamma_x \oplus L_l,$$
e.g. take the orthogonal complement to $\gamma_x$ in $\lambda$ 
by a fixed Hermitian metric of $\C^4$. 
That yields an isomorphism of (topological) vector bundles 
$$\bar{h}^*S \simeq \pi^*\gamma \oplus L.$$ 
To each $(x, l) \in F_3$, 
we associate the viewpoint 
\begin{equation}\label{q}
q=q_{(x,l)}:=[L_l] \in \Proj^3
\end{equation}
and the projection 
$$pr_q: \Proj^3 -\{q\} \to \Proj_{(x,l)}^2:=P(\gamma_x\oplus V_l).$$
Recall that a stable map $f: M \to X \subset \Proj^3$ is given. 
Let $p \in M$ and $x=f(p)$. 
Given a line $l\in \Proj^3$, 
we can define a holomorphic map-germ at $p$ 
\begin{equation}\label{varphi}
\varphi_{p,l}:(M,p)\longrightarrow (\Proj_{(x,l)}^2, x)
\end{equation}
by composing $pr_q$ and  $f: M \to \Proj^3$. 
Note that the tangent space of the screen projective plane is identified as follows: 
$$T_x \Proj^2_{(x,l)}=T_xP(\gamma_x\oplus V_l) \simeq \gamma_x^* \otimes V_l.$$
Thus we get a family of holomorphic map-germs 
$$
\begin{tikzcd}
g^*TM \arrow{rd}\arrow{rr}{\varphi}
& &
g^*\gamma^* \otimes h^*V \arrow{ld}  \\
& F_M  & 
\end{tikzcd}
$$
For short, 
we denote by  
 $TM$ and $\gamma^*\otimes V$ 
the pullback $g^*TM$ and $g^*(\gamma^*) \otimes h^*V$ over $F_M$, respectively. 

Notice that the family $\varphi$ is admissible 
for a generic surface $X$ in Theorem \ref{proj23} 
and for a generic choice of our viewpoints $q_{(x, l)}$ lying on $l$. 
That follows from the proof of  Theorem \ref{proj23} 
as mentioned in Remark \ref{Platonova}. 
The projection from any $q\; (\not=x) \in l$ yields 
a holomorphic family $\varphi_{(x,l,q)}$ 
over the space of $(x, l, q)$. 
By the transversality of the jet extension of $\varphi_{(x,l,q)}$ with $\A$-orbits 
(as discussed in \cite{Platonova, Kabata}), 
the space of $(x, l, q)$ is stratified according to the $\A$-types. 
Then it suffices to choose 
a continuous section $(x, l) \mapsto (x, l, q_{(x,l)})$ 
to be transverse to the strata. 

Let $X$ be generic as in Theorem \ref{proj23}. 
For each $\A$-type $\eta$ of  
plane-to-plane map-germs (Table \ref{table_2to2}), set 
$$\eta(\varphi):=\{\,(x,l)\in F_M\ |\ \varphi_{p,l}\sim_{\A}\eta\,\}.$$  
Note that 
for each $\eta$ except for Goose, 
the closure of $\eta(\varphi)$ does not depend 
on the choice of $q_{(x,l)}$.  
Projecting it to $X$, we have the $\eta$-singularity locus of $X$ 
as explained in \S 3.1: 
\begin{itemize} 
\item the locus of Lips/Beaks = the parabolic curve; 
\item the locus of Swallowtail = the flecnodal curve; 
\item the locus of Butterfly = degenerate flecnodal points; 
\item the locus of Gulls = cusps of Gauss; 
\item the locus of corank $2$ singularity = crosscaps; 
\item the locus of Goose  (in the above sense) 
depends on our choice of $q_{(x,l)}$, see 
Remark \ref{goose}. 
\end{itemize}
Then by Theorem \ref{main_thm}  we have 
$$\Dual[\overline{\eta(\varphi)}]=Tp_{\A}(\eta)(\bc,\bc')\in H^*(F_M)$$
where 
$$\bc=c(TM), \quad \bc'=c(\gamma^*\otimes V).$$

Let us compute the degree of the universal polynomial. 
Set 
$$t=c_1(L^*) \in H^2(P(Q)),$$ 
then it is well-known that 
	\begin{align*}
		H^*(P(Q))&=H^*(\Proj^3)[t]/\langle t^3+c_1(Q)t^2+c_2(Q)t+c_3(Q)\rangle \\
		&=\Z[a,t]/\langle a^4,t^3+at^2+a^2t+a^3\rangle. 
	\end{align*}
Note that 
\begin{align*}
	c(\bar{h}^*S)&=c(\gamma\oplus L)=(1-a)(1-t),\\
	c(\bar{h}^*V)&=\frac{1}{c(\bar{h}^*S)}= 1+(a+t)+(a^2+t^2+at)+\cdots, 
\end{align*}
and hence 
\begin{align*}
	c(\gamma^*\otimes V)&=(1+\alpha+a)(1+\beta+a) \quad \text{($\alpha, \beta$: Chern roots)}\\
	&=1+3a + t+3a^2+2at+t^2. 
	\end{align*}
So $Tp_{\A}(\eta)(\bc,\bc')$ in $H^*(F_M)$ 
is written by $c_i(TM)$ together with $a$ and $t$. 
Push it down via $\pi \circ i=f\circ g: F_M \to \Proj^3$ 
by using well-known formulae 
$$\pi_*(t^2)=1, \;\; \pi_*(t^3)=-a, \;\; \pi_*(t^4)=\pi_*(t^5)=\cdots=0,$$ 
then we get the degree in terms of $\xi_1$, $\xi_2$, $\xi_{01}$ and $a$.  
We summarize the result as follows:

\begin{thm} \label{dCT}
For a generic surface $X$ of degree $d$ in $\Proj^3$ with ordinary singularities, 
we take a stable map $f: M \to \Proj^3$ with $X=f(M)$. 
Then the degree of the locus having singular projection 
with prescribed type are expressed by the Chern numbers $d, \xi_1, \xi_2, \xi_{01}$ 
associated to $f$ as in Table \ref	{singular_projection_P3}. 
\end{thm}

\begin{table}
		\centering
        \begin{tabular}{|c|l|}\hline 
	        Type &  Degree \\ \hline \hline
	       Parabolic curve  & $8d-4\xi_1$ \\ \hline
	       Flecnodal curve & $20d-11\xi_1$ \\ \hline
	       Deg. flecnodal pt &$5 (30 d - 5 \xi_{01} + 12 (-3 \xi_1 + \xi_2))$ \\ \hline
	        Cusp of Gauss & $62 d + 3 \xi_{01} - 72 \xi_1 + 19 \xi_2$ \\ \hline
	        Crosscap &$6d-\xi_{01}-4\xi_1+\xi_2$  \\ \hline
        \end{tabular} \vspace{8pt}
       	\caption{Degree of loci on a surface in $\Proj^3$.}
	\label{singular_projection_P3}
\end{table}

By  Lemma \ref{Enriques}, we have the following expressions 
in terms of $d, C, \epsilon_0, T$: 

\begin{cor}
	\label{ordsing2to3}
	The degrees as in Theorem \ref{dCT} are also expressed by 
	four numerical characters such as 
	the degree $d$ of $X$, the number of crosscaps $C$, the number of triple points $T$, 
	and the degree of double curve $\epsilon_0$ in  Table \ref{classical_formulae}. 
\end{cor}
	\begin{table}
		\centering
        	\begin{tabular}{|c|l|}\hline 
        		Type &  Degree \\ \hline \hline
        		Parabolic curve  & $4d(d-2)-8\epsilon_0$ \\ \hline
        		Flecnodal curve & $d(11d-24)-22\epsilon_0$ \\ \hline
        		Deg. flecnodal pt  &$5d(d-4)(7d-12)-10 C + 105 T+5\epsilon_0(80-21d)$ \\ \hline
        		Cusp of Gauss  & $2d(d-2)(11d-24)-25C+66T+\epsilon_0(184-66d)$ \\ \hline
       			Crosscap & $C$ \\ \hline
         	\end{tabular} \vspace{8pt}
         	\caption{Expression in $d$, $T$, $C$ and $\epsilon_0$. }
	\label{classical_formulae}
	\end{table}
Table \ref{classical_formulae} generalizes some classical formulae 
of Salmon-Cayley-Zeuthen for smooth surfaces in $\Proj^3$, 
see Remark \ref{Kulikov} below.

\begin{rem}\upshape \label{goose}
{\bf (Tp for Goose)} 
In entirely the same way as above, 
we compute the degree for $Tp_\A$ of Goose type 
and obtain the number 
$$22 d - \xi_{01} - 24 \xi_1 + 7 \xi_2 =
2d(d-2)(3d-8) -5C+18T+\epsilon_0(56-18d).$$ 
It counts the following special parabolic points. 
As mentioned in \S 3.1, 
at each parabolic point $x$,  
there is a unique viewpoint on the asymptotic line $l$, say $\tilde{q}_{(x,l)} \in l$, 
from which the germ at $x$ of 
the central projection is of the Goose type 
(Lips/Beaks is observed from any other points on $l$). 
In fact, 
the locus of special viewpoints $\tilde{q}_{(x,l)}$ in $\Proj^3$ is a space curve 
formed by the cuspidal edge of the {\it scroll} (ruled surface) 
consisting of all asymptotic lines over parabolic curves \cite{Arnold, Platonova}.   
On the other hand, in our construction of the family $\varphi$, 
we have chosen the viewpoint $q_{(x,l)}$ 
at the `infinity with respect to $x$'. 
So we have different two sections $q_{(x,l)}$ and $\tilde{q}_{(x,l)}$ 
of the projective line bundle $P(\lambda)$ over the parabolic curve  
whose fibers are the asymptotic lines. 
The intersection number of these two sections is just 
the above degree of $Tp_\A$ of Goose. 
\end{rem}

\begin{rem}\upshape \label{Kulikov}
{\bf (Classical formulae)} 
For a generic smooth surface in $\Proj^3$ of degree $d$, 
expressions in terms of $d$ for  
the degrees of  parabolic and flecnodal curves and 
the number of cusps of Gauss  are classically well-known,  
see  Salmon \cite[p.559, p.602]{Salmon} 
-- those are recovered in Table \ref{classical_formulae} with $T=C=\epsilon_0=0$. 
In 80's, following  the classification of projections \cite{Arnold, Platonova}, 
Kulikov \cite{Kulikov} rediscovered  
the degree of the flecnodal curve and the number of the Butterfly locus 
by applying Thom polynomials of stable $A_k$-singularities (\cite{Porteous})  
to the map $h:F_M \to G_3$ between $4$-folds. 
However, the parabolic curve could not be dealt in the same way. 
Our approach is close to Kulikov's but more consistent; 
degrees of all singularity loci are simply obtained 
in a unified way using  $Tp_\A$ of {\it unstable} $\A$-types. 
\end{rem}

\subsection{Surfaces in $\Proj^4$} 
We extend the above argument to the following situation. 
Let $X$ be a generic surface of degree $d$ with transverse double points in $\Proj^4$, 
and $f:M\rightarrow \Proj^4$ an immersion with $X=f(M)$ and $M$ smooth,  
which resolves the double points.  
Put $G_4:=G(2,5)=G(\Proj^1,\Proj^4)$.  
We set the flag manifold to be 
$$F_4=\{(x,l)\in \Proj^4\times G_4\ |\ x\in l\},$$ 
and set the pullback 
$$F_M=\{(p,l)\in M\times G_4\ |\ f(p)\in l\}$$ 
with the diagram of projections and inclusions similar to the previous one. 

Take the Euler sequence over $\Proj^4$: 
\begin{equation}
\label{exactsq:P4}
\begin{tikzcd}
0 \arrow{r}&\gamma\arrow{r}&\epsilon^5_{\Proj^4}\arrow{r}&Q\arrow{r}& 0
\end{tikzcd}	
\end{equation}
where  $\gamma$ is the tautological line bundle, and then 
$P(Q) = F_4$. 
Over the Grassmaniann $G_4$, 
let $S$ denote the tautological bundle of rank $2$ 
and $V$ the quotient bundle of rank $3$. 
In entirely the same way as in the previous subsection, 
we define 
$\varphi: g^*TM \to g^*\gamma^*\otimes h^*V$, 
as a family of holomorphic map-germs $(\C^2,0)\rightarrow (\C^3,0)$ 
over $F_M$. 

We set $a:=c_1(\Ost_{\Proj^4}(1))$ and 
$t:=c_1(\Ost_{P(Q)}(1))$.  
Note that 
\begin{align*}
H^*(P(Q))
&=\Z[a,t]/\langle a^5,t^4+at^3+a^2t^2+a^3t+a^4\rangle.  
\end{align*}
A simple computation shows that 
\begin{align*}
	c_1(\gamma^*\otimes V)&=4a+t,\\
	c_2(\gamma^*\otimes V)&=6a^2+3at+t^2,\\
	c_3(\gamma^*\otimes V)&=4 a^3 + 3 a^2 t + 2 a t^2 + t^3. 
\end{align*}
These data are substituted for $c_j'$ in $Tp_\A(\eta) \in H^*(F_M)$ (Table \ref{table_tp2to3}) 
for each $\A$-type $\eta$ appearing in Theorem \ref{proj24},  
and then we push it out to $\Proj^4$. 
Note that  $\pi_*(t^3)=1$, $\pi_*(t^4)=-a$, $\pi_*(t^5)=\cdots = 0$. 
Put  
$$f_*(1)= d  a^2, \;\; f_*c_1(TM)=\xi_1  a^3,$$
$$f_*(c_1(TM)^2)=\xi_2 a^4, \;\; f_* c_2(TM)=\xi_{01} a^4.$$ 
We summarize the result as follows: 

\begin{thm} \label{dCT2}
For a generic immersed surface $X$ of degree $d$ in $\Proj^4$, 
we take an immersion $f: M \to \Proj^4$ with $X=f(M)$. 
Then the degree of the locus having singular projection 
with prescribed type are expressed by the Chern numbers $d, \xi_1, \xi_2, \xi_{01}$ 
associated to $f$ as in Table \ref{singular_projection_P4}. 
\end{thm}

	\begin{table}
		\centering	        
			\begin{tabular}{|c|l|}\hline 
		        Type &  Result \\ \hline \hline
		        $B_2$  & $25d-16\xi_1$\\ \hline
		        $H_2$ &$10d-6\xi_1$\\ \hline
		        $H_3$ & $5(42d-11\xi_{01}-51\xi_1+19\xi_2)$ \\ \hline
		        $P_3$ &$80d-15\xi_{01}-95\xi_1+33\xi_2$ \\ \hline
	        \end{tabular} \vspace{8pt}
	       	\caption{Degree of loci on a surface in $\Proj^4$.}
	        \label{singular_projection_P4}
	\end{table}

\begin{cor}\label{degree2to3}
	In case that $M=X$ is a smooth complete intersection of degree $d_1$ and $d_2$ in $\Proj^4$, 
	 degrees of the loci having singular projections are expressed in terms of $d_1, d_2$ as in Table \ref{complete_int}. 
\end{cor}

\begin{table}
		\centering
    		\begin{tabular}{|c|l|}\hline 
		        Type &  Result \\ \hline \hline
		        $B_2$  & $d_1d_2(16d_1+16d_2-55)$\\ \hline
		        $H_2$ &$d_1d_2(6d_1+6d_2-20)$\\ \hline
		        $H_3$ & $5d_1d_2(8{d_1}^2+8{d_2}^2+27d_1d_2-84d_1-84d_2+152)$\\ \hline
		        $P_3$ &$d_1d_2(18{d_1}^2+18{d_2}^2+51d_1d_2-160d_1-160d_2+280)$ \\ \hline
    		\end{tabular} \vspace{8pt}
		\caption{Expression by $d_1, d_2$ for a complete intersection surface in $\Proj^4$. }
		\label{complete_int}
\end{table}

\subsection{Primals in $\Proj^4$}
Numerical projective characters of a $3$-fold $X$ in $\Proj^4$ (primals) 
were studied by L. Roth  (cf. \cite{SR}). 
In our another paper \cite{SO}, 
we have reconsidered  Roth's work 
by applying Thom polynomials to generic projections; 
for instance, analogous formulae to Lemma \ref{Enriques} are easily obtained. 
In the present paper, as seen above, 
we are concerned with counting the number of non-generic projections; 
seemingly, that has not been treated for $3$-folds in any literature so far. 
We give only a partial result for loci with codimension one of non-generic projections. 
Assume that $X$ is smooth and of degree $d$. 

The flag manifold $F_X$ of pairs $(x, l)$ of points and lines 
with $x \in l$  is now of dimension $6$, 
and there is associated a family of holomorphic map-germs 
$\varphi: g^*TX \to g^*\gamma^*\otimes h^*V$ over $F_X$. 
Chern classes $c_i(\gamma^*\otimes V)$ have been computed in 3.2, and 
$c(TX)=(1+a)^5(1+d a)^{-1}$. 
Substituting them to $\bc, \bc'$ in $Tp_\A(\eta)$ in Table \ref{table_tp3to3}, 
we obtain the following:

\begin{thm} 
For a generic smooth hypersurface $X$ of degree $d$ in $\Proj^4$, 
the degree of the locus having singular projection 
with type $A_3, A_4, C, D$ are expressed in Table \ref{degreeP4}. 
\end{thm}

\begin{table}
		\centering
    		\begin{tabular}{|c|l|}\hline 
		        Type &  Degree \\ \hline \hline
		        $A_4$  & $10d(5d-12)$\\ \hline
		        $C$ & $5d(d-2)$\\ \hline
		        $D$ & $10d(4d-9)$\\ \hline
    		\end{tabular} \vspace{8pt}
		\caption{Degree of codimension $1$ loci on a $3$-fold in $\Proj^4$.}
		\label{degreeP4}
\end{table}

\begin{rem}\upshape
Since we assume that $X$ is smooth, any projection has singularities of corank at most $1$, 
so $Tp_\A(I_{2,2})$ is zero. 
Also a computation shows that 
the degree of the pushforward of $Tp_\A(A_3)$ is equal to $6d$. 
In fact, 
there are exactly six degenerate asymptotic lines at any general points $x \in X$ 
(lines having at least $4$-point contact with the hypersurface); 
all asymptotic lines at $x$ form a conic in $\Proj(T_xX) \simeq \Proj^2$, 
and moreover, degenerate asymptotic lines are given by one more equation of degree three 
(cubical form of the Monge form), 
thus by the Bezo\'ut theorem, generically six solutions are there. 
\end{rem}



\begin{thebibliography}{99999}
\bibitem{Arnold} V.~I.~Arnold, V.~V.~Goryunov, O.~V.~Lyashko, V.~A.~Vasil'ev,
{\it Singularity Theory II, Classification and Applications},
Encyclopaedia of Mathematical Sciences Vol. 39, Dynamical System VIII (V. I. Arnold (ed.)),
(translation from Russian version), Springer-Verlag  (1993).
%
\bibitem{Baker} H.~F.~Baker, 
{\it Principles of Geometry, VI. Introduction to the theory of algebraic surfaces and higher loci}, 
Cambridge University Press (1933). 
%
%
\bibitem{DK} J.~L.~Deolindo Silva and Y.~Kabata, 
Projective classification of jets of surfaces in $4$-space 
and applications, preprint  (2015); ArXiv:1601.06255. 
%
\bibitem{FR02} L.~Feh\'er and R.~Rim\'anyi, 
Calculation of Thom polynomials 
and other cohomological obstructions for group actions, 
{\it Real and Complex Singularities} (Sao Carlos, 2002), 
Contemp. Math.  {\bf 354}, A.M.S. (2004), 69--93.
%
\bibitem{FR} L.~Feh\'er and R.~Rim\'anyi, 
Thom series of contact singularities, 
Ann.  Math. {\bf 176} (2012), 1381-1426. 
%
\bibitem{FP} W.~Fulton and P.~Pragacz, 
{\it  Schubert varieties and degeneracy loci}, 
Lecture note in Math. {\bf 1689}, Springer-Verlag (1998). 
%
\bibitem{HK} A.~Haefliger and A.~Kosinski, 
Un th\'eor\`eme de Thom sur les singularit\'es des applications diff\'erentiables,
S\'eminaire H.~Cartan, ENS., 1956/57, Expos\'e no.~8.
%
\bibitem{Hawes}W.~Hawes, 
Multi-dimensional motions of the plane and space, 
PhD thesis, University of Liverpool (1994).
%
\bibitem{Kabata} Y. Kabata,
Recognition of plane-to-plane map-germs, 
Topology and its Applications {\bf 202} (2016), 216--238. 
%
\bibitem{Kaz03} M.~E.~Kazarian,  
Multisingularities, cobordisms and enumerative geometry, 
Russian Math. Survey {\bf 58} (2003), 665--724 (Uspekhi Mat. nauk {\bf 58}, 29--88). 
%
%
\bibitem{Kulikov} V.~S.~Kulikov,  
Calculation of singularities of an imbedding of a generic algebraic surface in projective space $\Proj^3$,  
Funct. Anal. Appl.  {\bf 17} (1983), 176-186
%
\bibitem{Kleiman} S.~L.~Kleiman,  
{\sl The enumerative theory of singularities}, 
Proc. Nordic Summer School/NAVF, Symposium in Math., Oslo, 
Sijthoff  and Noordhoff Inter. Publ. (1976), 297--396. 
%
\bibitem{Landis} E.~E.~Landis, 
Tangential singularities, Funt. Anal. Appl. {\bf 15} (1981), 103--114 (translation). 
%
\bibitem{MararTari} W.~L.~Marar and F.~Tari, 
On the geometry of simple germs of co-rank $1$ maps from $\R^3$ to $\R^3$, 
Math. Proc. Cambridge. Philos. Soc. {\bf 119}  (1996), no. 3, 469-481.
%
\bibitem{Mond82} D.~Mond, 
Classifcation of certain singularities and applications to differential geometry, 
Ph.D. thesis, University of Liverpool (1982).
%
\bibitem{Mond} D.~Mond, 
On the classification of germs of maps from $\R^2$ to $\R^3$, 
Proc. London Math. Soc. {\bf 50} (1984), 333--369.
%
\bibitem{Anna} A.~C.~Nabarro, 
Sobre a gemeotria local de hipersuperf\'icies em $\R^4$, 
PhD Thesis, USP - S\~ao Carlos (2000). 
%
\bibitem{NT} J.~Nu\~no-Ballesteros and F.~Tari, 
Surface in $\R^4$ and their projections to $3$-spaces, 
Proc. Royal Soc. Edinburgh Sect. A {\bf 137} (2007), 
%
%
%
\bibitem{Ohmoto14}T.~Ohmoto, 
Singularities of maps and characteristic classes,  
{\it School on Real and Complex Singularities in S\~ao Carlos, 2002}, 
Adv. Studies. Pure Math. {\bf 68} (2016), 191--265. 
%
\bibitem{Piene} R.~Piene, 
Some formulae for a surface in $\Proj^3$, 
{\it Algebraic Geometry}, 
Lecture Notes Math. {\bf 687}, Springer  (1978), 196--235. 
%
\bibitem{Platonova} O.~A.~Platonova, 
Projections of smooth surfaces, 
J. Soviet Math {\bf 35} (1986), 2796--2808 
%
\bibitem{Porteous} I.~R.~Porteous, 
Simple singularities of maps,  
Proceedings of Liverpool Singularities — Symposium I, 
Lecture Notes in Math. {\bf 192}, Springer (1971), 286--307. 
%
\bibitem{Rieger} J.~H.~Rieger, 
Families of maps from the plane to the plane, 
J. London Math. Soc.  {\bf 36} (1987), 351--369. 
%
\bibitem{Rimanyi} R.~Rim\'anyi, 
 Thom polynomials, symmetries and incidences of singularities, 
Invent. Math. {\bf 143} (2001), 499-521. 
%
%
\bibitem{Salmon} G.~Salmon, 
{\it A treatise on the analytic geometry of three dimensions},  
4th edition, Dublin (1882). 
%
\bibitem{SKSO} H.~Sano, Y.~Kabata, J.~L.~Deolindo Silva and T.~Ohmoto, 
Projective classification of jets of surfaces in $3$-space, preprint (2015);  
ArXiv:1504.06499. 
%
%
\bibitem{SO} T.~Sasajima and T.~Ohmoto, 
Classical formulae on projective characters of surfaces and $3$-folds, revisited, 
preprint (2016); ArXiv:1606.09138. 
%
\bibitem{SR} J.~G.~Semple and L.~Roth, 
{\it Introduction to Algebraic Geometry},  
Oxford University Press (1949). 
%
\bibitem{UV} R.~Uribe-Vargas,
A projective invariant for swallowtails and godrons, 
and global theorems on the flecnodal curve,
Mosc. Math. Jour. {\bf 6} (2006), 731--772.
%
\bibitem{Thom} R.~Thom, 
Les singularit\'es des applications diff\'erentiables,  
Ann. Inst. Fourier {\bf 6} (1955--56), 43--87.
%
\bibitem{West} J. ~West, 
The Differential Geometry of the Cross-Cap, 
Ph.D. thesis, University of Liverpool (1995). 
%
\bibitem{YKO} T.~Yoshida, Y.~Kabata and T.~Ohmoto, 
Bifurcations of plane-to-plane map-germs of corank $2$, 
Quarterly J. Math. {\bf 66} (2015), 369--391. 
%
%
\end{thebibliography}
\end{document}